\documentclass[11pt]{article}
\usepackage{latexsym,amssymb,amsmath,amssymb}
\usepackage{amsfonts,amsbsy,amsthm}
\usepackage[all]{xy}
\usepackage{graphicx}
\numberwithin{equation}{section}
\newtheorem{theorem}{Theorem}[section]
\newtheorem{lemma}[theorem]{Lemma}
\newtheorem{proposition}[theorem]{Proposition}
\newtheorem{corollary}[theorem]{Corollary}

\theoremstyle{definition}
\newtheorem{definition}[theorem]{Definition}

\newtheorem{remark}[theorem]{Remark}

\newcommand{\fun}[3]{#1: {#2} \rightarrow {#3}}
\newcommand{\gen}[1]{\langle{#1}\rangle}
\begin{document}


\topmargin3mm
\hoffset=-1cm
\voffset=-1.5cm
\

\medskip
\begin{center}
{\large\bf The rainbow connection number of enhanced power graph
}
\vspace{6mm}\\
\footnotetext[2]{This work was partially supported by CONACYT.}

\end{center}

\medskip
\begin{center}
Luis A. Dupont, Daniel G. Mendoza and Miriam Rodr\'{\i}guez.
\\
{\small Facultad de Matem\'aticas, Universidad Veracruzana}\vspace{-1mm}\\
{\small Circuito Gonzalo Aguirre Beltr\'an S/N;}\vspace{-1mm}\\
{\small Zona Universitaria;}\vspace{-1mm}\\
{\small Xalapa, Ver., M\'exico, CP 91090.}\vspace{-1mm}\\
{\small e-mail: {\tt ldupont@uv.mx}\vspace{4mm}}
\end{center}

\medskip

\begin{abstract}
Let $G$ be a finite group, the enhanced power graph of $G$, denoted by $\Gamma_G^e$, is the graph with vertex set $G$ and two vertices $x,y$ are edge connected in $\Gamma_{G}^e$ if there exist $z\in G$ such that $x,y\in\langle z\rangle$. Let $\zeta$ be a edge-coloring of $\Gamma_G^e$. In this article, we calculate the rainbow connection number of the  enhanced power graph $\Gamma_G^e$.

\end{abstract}

\section{Introduction}
Let $G$ be a finite group, the power graph of a finite group $G$ we denote the power graph by $\Gamma_G$, it is the graph whose vertex set are the elements of $G$ and two elements being adjacent if one is a power of the other. In \cite{Aalipour} the authors found that the power graph is contained in the non-commuting graph and, they asked about how much the graphs are closer, and then, they defined the \emph{enhanced power graph} of a finite group. We denoted to the enhanced power graph by $\Gamma_G^e$ whose vertex set is the group $G$ and two distinct vertices $x,y\in  V(\Gamma_G^e)$ are adjacent if $x,y\in\gen{z}$ for some $z\in G$. Later, the enhanced power graph of a group was studied by Sudip Bera and A. K. Bhuniya~\cite{Bera-Bhuniya}.\\

In 2006, Chartrand, Johns, McKean and Zhang \cite{Chartrand} introduced the concept of rainbow connection of graphs. This concept was motivated by communication of information between agencies of USA government after the September 11, 2001 terrorist attacks. The situation that helps to unravel this issue about communications has as graph-theoretic model the following. Let $\Gamma$ be a connected graph with vertex set $V(\Gamma)$ and edge set $E(\Gamma)$. We define a coloring $\fun{\zeta}{E(\Gamma)}{\{1,...,k\}}$ with $k\in\mathbb{N}$. A path $P$ is a \emph{rainbow} if any two edges of $P$ are colored distinct. If for each pair of vertices $u,v\in V(\Gamma)$, $\Gamma$ has a rainbow path from $u$ to $v$, then $\Gamma$ is  \emph{rainbow-connected} under the coloring $\zeta$, and $\zeta$ is called a \emph{rainbow k-coloring} of $\Gamma$. The \emph{rainbow connection number of $\Gamma$}, denoted by $rc(\Gamma)$ is the minimum $k$ for which there exists a rainbow $k$-coloring of $\Gamma$.\\

We will apply the idea of calculating the rainbow connected number of enhanced power graph through the graphs such that as was carried out by the authors from  \cite{Ma} about the power graph, with $InvMax_G$, the set of maximal involution of $G$, whose important theorems we can summarize in the following:

\begin{theorem}
Let $|InvMax_G|\neq\emptyset$ and $G$ be a finite group of order at least 3. Then
\begin{center}
  $rc(\Gamma_G)=
\begin{cases}
3, & \text{if } 1\leq |InvMax_G|\leq 2; \\
|InvMax_G|, & \text{if } |InvMax_G|\geq 3.
\end{cases}$
\end{center}
If $|InvMax_G|=\emptyset$, let $G$ be a finite group
\begin{enumerate}
  \item If $G$ is cyclic, then
$  rc(\Gamma_G) =
\begin{cases}
1, & \text{if }|G|\text{ is a prime power}; \\
2, & \text{otherwise}
\end{cases}$
  \item If $G$ es noncyclic, then $rc(\Gamma_G)= 2$ or $3$.
\end{enumerate}
\end{theorem}

In this paper we compute the rainbow connection number of $\Gamma_G^e$ and we characterize it in terms of independence cyclic set, whose particular case is maximal involution. This paper is organized as follows. In section 2 we put definitions and some properties about rainbow connection number and we describe a way for guarantee a coloring for enchanced power graphs. In section 3 we wrote the main theorems for determine $\Gamma_G^e$.

\section{Definitions and properties}
We start the section with a proposition from enhanced power graph definition.
\begin{proposition}\label{prop1}
  $rc(\Gamma_G^e)=1$ if only if $\Gamma_G^e$ is complete if only if $G$ is cyclic.
\end{proposition}

\begin{definition}
Let $Max_G=\{x_1,...,x_m\}$ be an \emph{essential cyclic set} if
\begin{enumerate}
  \item for all $g\in G$, $\gen{g}=\gen{x_i}$ for some $i$,
  \item $\gen{x_i}\neq\gen{x_j}$ for $i\neq j$,
  \item each $x_i$ is a maximal cyclic subgroup.
\end{enumerate}
\end{definition}

Therefore \ref{prop1} can be rewritten as follows
\begin{proposition}\label{prop23}
$|Max_G|=1$ if only if  $G$ is a cyclic group if only if $rc(\Gamma_G^e)=1$
\end{proposition}

\begin{proposition}
  If $|Max_G|=2$, then $rc(\Gamma_G^e)=2$
\end{proposition}
\begin{proof}
  Since $\Gamma_G^e$ is not complete, we have $rc(\Gamma_G^e)\geq 2$, then we have
  \begin{center}
  $\begin{array}{lll}
  E_1&=&\big\{\{a,b\}|a,b\in\gen{x_1}\big\}\\
  E_2&=&\big\{\{a,b\}|a,b\in\gen{x_2}\big\}
  \end{array}$
  \end{center}
  We can note that the only one path  between $x_{1_j}$ and  $x_{2_i}$ for all $x_{1_j}\in\gen{x_1}$ and $x_{2_i}\in\gen{x_2}$ is $(x_{1_j},e,x_{2_i})$, then the 2-coloring is given by $\zeta: E(G)\longrightarrow \{1,2\}$ with $f\mapsto i$, if $f\in E_i$ is a rainbow 2-coloring of $\Gamma_G^e$.
\end{proof}

\begin{definition}
  We define the \emph{independence cyclic set of $Max_G$}, denoted by $ics(G)$, as
  $$ics(G)=\{x_i\in Max_G|\gen{x_i}\cap\gen{x_j}=e  \text{ for } i\neq j\}$$
  The  independence cyclic number of $Max_G$, denoted by $icn(G)$, is $icn(G)=|ics(G)|$.
\end{definition}

\begin{remark}\label{contencion}
We note that $$InMax_G\subseteq ics(G)\subseteq Max_G.$$
\end{remark}

\begin{proposition}\label{star}
  If $|Max_G|=3$, then $$rc(\Gamma_G^e)=\begin{cases}
                                   2, & \mbox{if } icn(G)=1 \\
                                   3, & \mbox{if } icn(G)=3
                                 \end{cases}$$
\end{proposition}
\begin{proof}
  Let $Max_G=\{x_1,x_2,x_3\}$ be an essential cyclic set.

\begin{remark}\label{edge}
 We do not need to be concise with the path with both vertex in $x_i$ for some $i$, because with one color, we can coloring this path. The difficult is when both vertex are in different $x_i$.
\end{remark}

\textbf{Case $\boldsymbol{cin(G)=1}$} Without loss of generality we suppose $\gen{x_1}\cap\gen{x_2}=e=\gen{x_1}\cap\gen{x_3}$ and $\gen{x_2}\cap\gen{x_3}\neq e$. Since $G$ is not cyclic group, then $rc(\Gamma_G^e)\geq 2$. Let $h\in\gen{x_2}\cap\gen{x_3}$ with $h\neq e$, thus
      \begin{center}
$\begin{array}{lll}
  E_1&=&\big\{\{a,b\}|\{a,b\}\subset\gen{x_1}\big\}\bigcup\big\{\{a,b\}|\{a,b\}\subset\gen{x_2}\text{ with }a,b\neq e\big\}\\
  E_2&=&\big\{\{e,g\}|g\in\gen{x_2}\cup\gen{x_3}\big\}\bigcup\big\{\{a,b\}|a\in\gen{x_3}\setminus\gen{x_2},\quad b\in\gen{x_2}\cap\gen{x_3},\,b\neq e\big\}
\end{array}$
  \end{center}
  In particular $\{h,g\}\in E_2$ for all $g\in\gen{x_3}\setminus\gen{x_2}.$ Then, we will give a 2-coloring to $\Gamma_G^e$:
  \begin{equation}\label{rc2}
    \begin{array}{ccc}
      \zeta:E(G) & \longrightarrow & \{1,2\} \\
       f& \mapsto & i
    \end{array} \text{ if } i\in E_i
  \end{equation}
\begin{figure}[!htb]
\centering
\includegraphics[scale=.2]{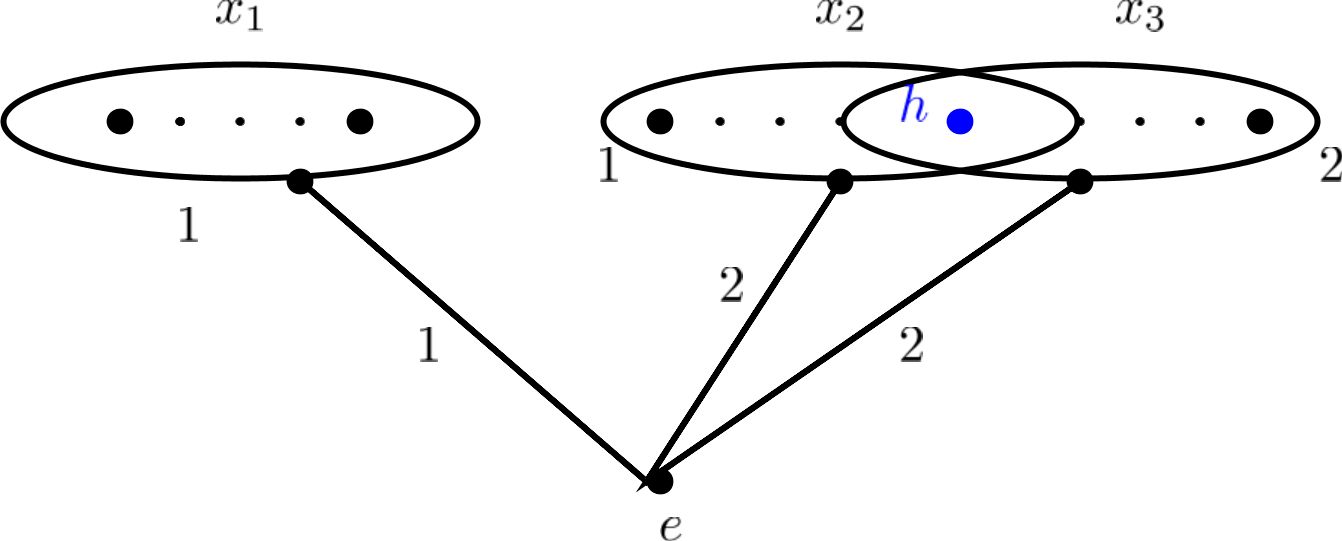}
\end{figure}

\textbf{Case $\boldsymbol{cin(G)=3}$} We suppose that $|InMax_G|=0$, and without loss of generality $\gen{x_i}\cap\gen{x_j}=e$ for $1\leq i<j\leq 3$. We will give a 3-coloring for $\Gamma_G^e$, with
$$\begin{array}{lll}
    E_1 & = & \big\{\{x_i,e\}|i=1,2,3   \big\} \\
    E_2& = & \big\{\{e,x_{i_j}\}|x_{i_j}\in\bigcup_{i=1}^{3}\gen{x_i}\setminus x_i  \big\} \\
    E_3 & = & \big\{\{a,b\}|a,b\in\gen{x_i}\text{ for }i=1,2,3  \big\}
  \end{array}$$

 With the coloring
   \begin{equation}\label{rc3}
    \begin{array}{ccc}
      \zeta:E(G) & \longrightarrow & \{1,2,3\} \\
       f& \mapsto & i
    \end{array} \text{if } i\in E_i
  \end{equation}
\begin{figure}[!htb]
\centering
\includegraphics[scale=.2]{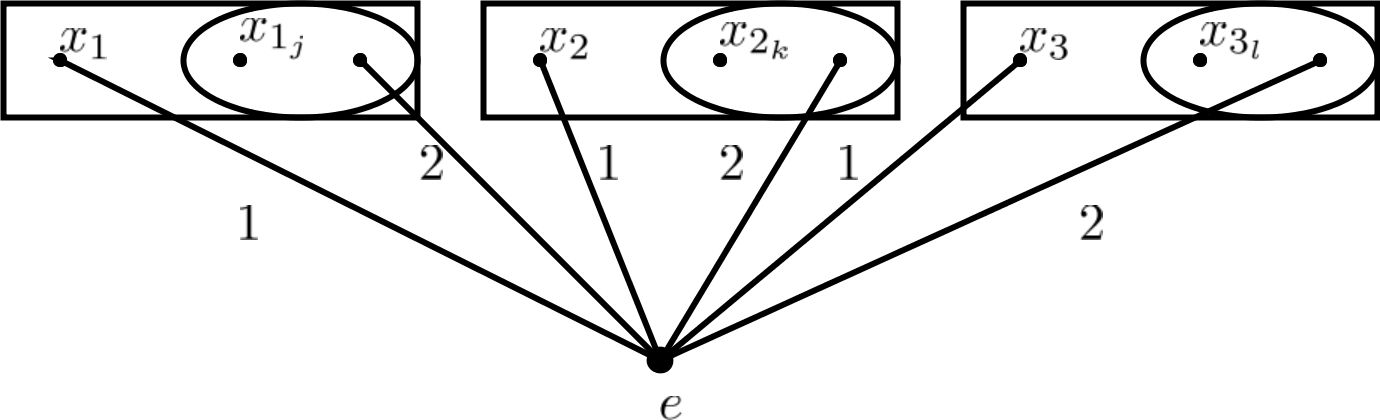}
\end{figure}

Now, we suppose that $InMax_G=Max_G$, then with $E_i=\big\{\{a,b\}|a,b\in\gen{x_i}\big\}$ be the edges set, and the coloring is given like \ref{rc3}.
\begin{figure}[!htb]
\centering
\includegraphics[scale=.2]{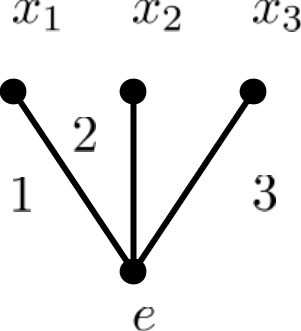}
\end{figure}

We can not give a 2-coloring for $\Gamma_G^e$. We claim that there is a 2-coloring. Let $u\in\gen{x_1}$, $v\in\gen{x_2}$ and $w\in\gen{x_3}$. Then, we have $\zeta(u,e)=1$ and $\zeta(e,v)=2$, thus $(u,e,v)$ is a desire rainbow path. Likewise $\zeta(u,e)=1$ and $\zeta(e,w)=2$, but for $(v,e,w)$ there is not a rainbow path.
\end{proof}

From \ref{star} we can ask ourself about what happens whether no one of $\gen{x_i}$ can be intersected by another $\gen{x_j}$ with $i\neq j$ or, what happens if all $\gen{x_i}$ are intersected with some common elements. For this, we have the following prepositions.\\

The following preposition is just like \cite[Proposition 2.4]{Ma}

\begin{proposition}\label{prop27}
  Let $Max_G=\{x_1,...,x_m\}$ be an essential cyclic set and $InMax_G=Max_G$. Then $rc(\Gamma_G^e)=m$.
\end{proposition}
\begin{proof}
For $\Gamma_G^e$ we will give a m-coloring. For each $i=1,...,m$ we have
$$E_i(G)=\big\{\{a,b\}|a,b\in\gen{x_i}\big\}$$
since for $u\in\gen{x_i}$ and $v\in\gen{x_j}$ with $i\neq j$ we have a only one path between them, which is $(u,e,v)$, and the coloring is given by
$$\begin{array}{rcl}
  \zeta: E(G) & \longrightarrow & \{1,...,m\} \\
  f & \mapsto & i\quad \text{if } f\in E_i
\end{array}$$
We can see the diagram in figure \ref{fig27}.
\begin{figure}[!htb]
\centering
\includegraphics[scale=.2]{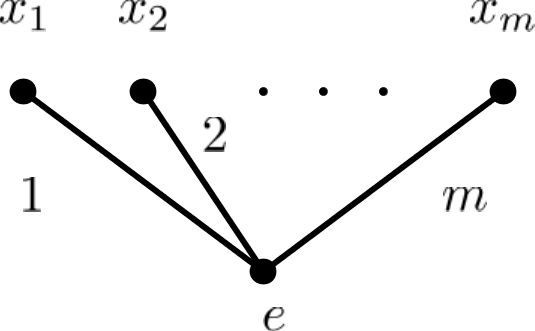}
\caption{$InMax_G=Max_G$}
\label{fig27}
\end{figure}
\end{proof}

\begin{proposition}\label{prop28}
Let $Max_G=\{x_1,...,x_m\}$ be an essential cyclic set with $m\geq 2$, and $h_{i,j}\in\gen{x_i}\cap\gen{x_j}$ for $1\leq i<j\leq m$. If $h_{i,j}\neq h_{r,s}$, with $i\neq r$ or $j\neq s$, then $rc(\Gamma_G^e)=2$.
\end{proposition}
\begin{proof}
By \ref{edge} we only give the coloring for $x_i$ and $x_j$ such that $i< j$. We fix
\begin{center}
$E_1(G)=\big\{\{a,h_{i,j}\}|a\in\gen{x_i}\setminus\gen{x_j}\big\}\bigcup\big\{\{a,b\}|a,b\in\gen{x_i}\big\}$\\
$E_2(G)=\big\{\{b,h_{i,j}\}|b\in\gen{x_j}\setminus\gen{x_i}\big\}\bigcup\big\{\{a,b\}|a,b\in\gen{x_j}\big\}$\\
\end{center}

Then, we always have a path for $x_{i_r}\in\gen{x_i}$ to $x_{j_s}\in\gen{x_j}$ given by $(x_{i_r},h_{i,j},x_{j_s})$ with $i<j$,  and the coloring is the same given in \ref{rc2}.
\end{proof}

\begin{figure}[!htb]
\centering
\includegraphics[scale=.15]{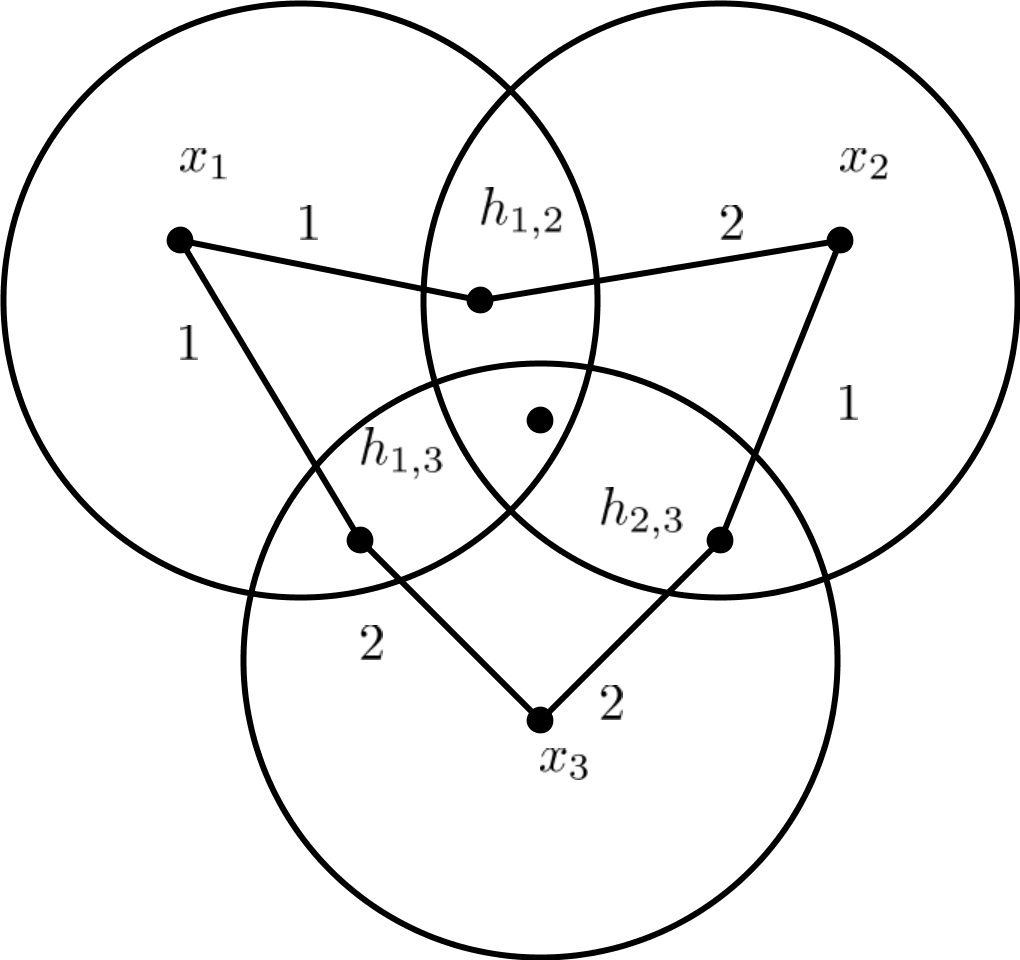}
\caption{Example for \ref{prop28} for $m=3$}
\end{figure}
The next definition guarantees the existence of a coloring for $\Gamma_G^e$.

\begin{definition}\label{awning}
  An \emph{awning} is a collection $H_1,...,H_{m-1}$ where the following occurs:
  \begin{enumerate}
    \item $H_i=A_i\mathop{\dot{\bigcup}} B_i=\{h_{i,i+1},...,h_{i,m}\}\subset\gen{X_i}$ for $i=1,...,m-1$
    \item for all $i<j$, $h_{i,j}\in\gen{x_i}\cap\gen{x_j}$
    \item for $i<j$ with $j=2,...,m-1$ , if $h_{j,s}=h_{i,r}\in H_j\cap H_i$ ($s\in\{j+1,...,m\}$, and $r\in\{i+1,...,m\}$), the following holds:
    \begin{enumerate}
      \item $r=j$, $h_{i,r}\in A_i$, then $h_{j,s}\in B_j$
      \item $r=j$, $h_{i,r}\in B_i$, then $h_{j,s}\in A_j$
      \item $r=s>j$, $h_{i,r}\in A_i$, then $h_{j,r}\in A_j$
      \item $r=s>j$, $h_{i,r}\in B_i$, then $h_{j,r}\in B_j$
    \end{enumerate}
  \end{enumerate}
\end{definition}

\begin{remark}\label{remawn}
The case in  \ref{prop27} is a particular case where $G$ has not an awning. By definition of awning we want to say, if we have an awning, then we only need $H_i=\{e\}$ for only some $i$, and no more.
\end{remark}

\begin{corollary}\label{cor1}
If $G$ has an awning and $|Max_G|\geq 3$, then $icn(G)\leq 1$. In particular, $|InMax_G|\leq 1$.
\end{corollary}
\begin{proof}
Suppose that $icn(G)=2$, then $H_{i_1}=H_{i_2}=\{e\}$. Hence $(x_{i_1},e,x_{i_2})$ is a rainbow path, and $(x_{i_1},e,x_{i_3})$ is another rainbow path, but in $(x_{i_2},e,x_{i_j})$ we have not a rainbow path for $\Gamma_G^e$.
\end{proof}

\begin{corollary}
  If $|\cap H_i|\geq m-1$ with $Max_G=m$, then $G$ has an awning.
\end{corollary}

\begin{corollary}\label{cor2}
If $|Max_G|=2$ then $icn(G)=0$ or $2$, and $G$ has an awning.
\end{corollary}

\begin{corollary}\label{cor15}
  If $G$ has an awning, then $icn(G)=1$. In particular $|InMax_G|\leq 1$.
\end{corollary}

\begin{remark}
  We note that the coloring whether we have to $InMax_G$ or $ics(G)$ does not change, both can be colored by only one color. The only one difference due to in $InMax_G$ there is only two elements in the subset of $G$ and, for a set taken of  $ics(G)$ there are more than two elements but, the behaviour in coloring is exactly the same, because, in a set taken of $ics(G)$ all the elements are associated each them, then, one color is enough for coloring all set.
\end{remark}

In the following properties we only consider the set $ics(G)$ unless otherwise indicated.
\begin{proposition}\label{prop}
  If $G$ has an awning, then $rc(\Gamma_G^e)=2$
\end{proposition}
\begin{proof}
  We will give to $\Gamma_G^e$ a rainbow 2-coloring, for $1\leq r < s \leq m$, let:\\
    \begin{center}
 $\begin{array}{lll}
  E_{r,s}^1&=&\{\{a,h_{r,s}\}|a\in\gen{x_r}\backslash\gen{x_s}; h_{r,s}\in A_r\}\\
  E_{r,s}^2&=&\{\{b,h_{r,s}\}|b\in\gen{x_s}\backslash\gen{x_r}; h_{r,s}\in A_r\}\\
  E_{r,s}^1&=&\{\{a,h_{r,s}\}|a\in\gen{x_r}\backslash\gen{x_s}; h_{r,s}\in B_r\}\\
  E_{r,s}^2&=&\{\{b,h_{r,s}\}|b\in\gen{x_s}\backslash\gen{x_r}; h_{r,s}\in B_r\}
  \end{array}$\\
 \end{center}
  Write $E_1=\displaystyle\bigcup_{1\leq r < s \leq m} E_{r,s}^1$ and $E_2=\displaystyle\bigcup_{1\leq r < s \leq m} E_{r,s}^2 $ and we define a coloring\\
    \begin{center}
 $ \begin{array}{rcl}
    \zeta: E(\Gamma_G^e) & \longrightarrow & \{1,2\} \\
    f & \mapsto & i,\quad \text{ if } f\in E_i
  \end{array} $
  \end{center}
We go to check that, this is a 2-coloring for $\Gamma_G^e$. We will make a coloring for $j,s-$step. If this edges have been colored in a before step, i.e., if $h_{j,s}=h_{i,r}$ with $i<j$, thus we will have coloring problems with $r=j$ or $r=s$.\\

For $r=j$ (r=s), for (a)-(d) from \ref{awning} we can guarantee in before step we can conserve the coloring and that, not affect us with the 2-coloring that we gave.
\end{proof}

\begin{proposition}\label{prop2}
If $rc(\Gamma_G^e)=2$, then for any order of $Max_G$, we have an awning.
\end{proposition}
\begin{proof}
We have $rc(\Gamma_G^e)=2$  and suppose $E_1\mathop{\dot{\bigcup}}E_2=E$ be the set of edges of $\Gamma_G^e$ and a 2-coloring give by \ref{rc2} and let $Max_G=\{x_1,...,x_m\}$ be an independence cyclic set of $\Gamma_G^e$, thus there is $h\in \gen{x_i}\cap\gen{x_j}$ such that $\{x_i,h\}\in E_1$ and $\{h,x_j\}\in E_2$ (or $\{x_i,h\}\in E_2$ and $\{h,x_j\}\in E_1$). We define $h_{i,j}:=h$, moreover $H_i:=\{h_{i,1},...,h_{i,m}\}=:A_i\mathop{\dot{\bigcup}}B_i$ such that
\begin{center}
  $A_i=\{h_{i,j}|\{x_i,h_{i,j}\}\in E_1\}$ and $B_i=\{h_{i,j}|\{x_i,h_{i,j}\}\in E_2\},$
\end{center}
where $(a)-(b)$ from \ref{awning} are met.
\end{proof}

\begin{corollary}
If $G$ has an awning with any order on $Max_G$, then for every order, $G$ has an awning.
\end{corollary}

\begin{theorem}\label{thm1}
$rc(\Gamma_G^e)=2$ if only if $G$ has an awning and $G$ is not cyclic group.
\end{theorem}
\begin{proof}
By \ref{prop} and \ref{prop2}
\end{proof}

By \ref{contencion} we obtain a similar proposition like \cite[Lemma 2.2]{Ma}.

\begin{lemma}
Let $Max_G=\{x_1,...,x_m\}$ be an essential cyclic set. If $InMax_G\neq\emptyset$, then $|InMax_G|\leq rc(\Gamma_G^e).$
\end{lemma}
\begin{proof}
As in the proof of \cite[Lemma 2.2]{Ma}.
\end{proof}

\begin{proposition}\label{prop221}
Let $Max_G=\{x_1,...,x_m\}$ be an essential cyclic set. If $icn(G)\geq 3$, then $3\leq rc(\Gamma_G^e).$
\end{proposition}
\begin{proof}

  Suppose that $|InMax|=0$ and $ics(G)=\{x_1,...,x_k\}$ be an independence cyclic set with $k\geq 3$. We can not give a 2-coloring for the graph induced by $\gen{x_1}\displaystyle\cup\cdots\cup\gen{x_k}$, but we will give a 3-coloring induced by the following edge sets
  $$\begin{array}{lll}
    E_1 & = & \big\{\{x_i,e\}|i=1,...,m   \big\} \\
    E_2& = & \big\{\{e,x_{i_j}\}|x_{i_j}\in\bigcup_{i=1}^{m}\gen{x_i}\setminus x_i  \big\} \\
    E_3 & = & \big\{\{a,b\}|a,b\in\gen{x_i}\text{ for each }i  \big\}
  \end{array}$$
\begin{figure}[!htb]
\centering
\includegraphics[scale=.15]{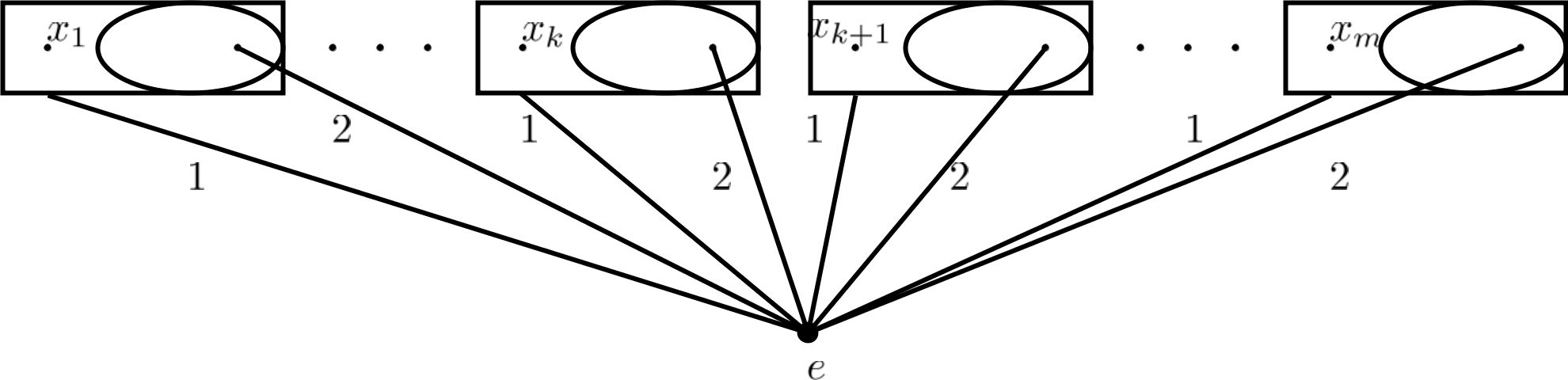}
\caption{$InMax_G=\emptyset$}
\end{figure}

  with the rest edges just like \ref{prop28} and \ref{prop}. Thus the 3-coloring is given by \ref{rc3}.\\

  If $|InMax_G|\geq 3$ then the edges set is
  $$\begin{array}{lll}
    E_1 & = & \big\{\{x_i,e\}|i=l+1,...,m   \big\} \\
    E_2& = & \big\{\{e,x_{i_j}\}|x_{i_j}\in\bigcup_{i=l+1}^{m}\gen{x_i}\setminus x_i  \big\} \\
    E_3 & = & \big\{\{a,b\}|a,b\in\gen{x_i}\text{ for each }i  \big\}\\
    E_i &=& \big\{\{x_i,e\}|i=1,...,l \big\}
  \end{array}$$
and the coloring given by
  $$\begin{array}{lll}
  \zeta:E(G)&\longrightarrow &\{1,...,l\}\\
  f&\mapsto & i
  \end{array} \text{ if }f\in E_i  $$
 \begin{figure}[!htb]
\centering
\includegraphics[scale=.15]{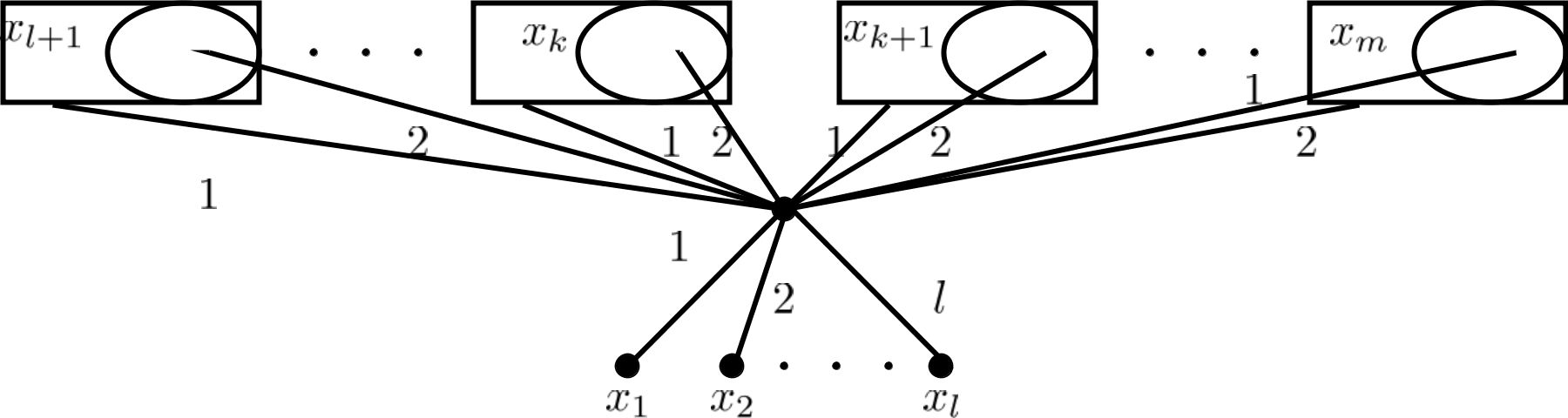}
\caption{$InMax_G\neq\emptyset$}
\end{figure}

\end{proof}
In particular we have the following

\begin{proposition}\label{prop3}
  Let $Max_G=\{x_1,...,x_m\}$ be a essential cyclic group with $m\geq 4$ and $icn(G)\geq 2$, then $3\leq rc(\Gamma_G^e)$.
\end{proposition}

\begin{remark}
We have $rc(\Gamma_G^e)\leq rc(\Gamma_G)$ because $E(\Gamma_G)\subseteq E(\Gamma_G^e).$
\end{remark}

\section{Main theorems}
In this section we prove our main theorems.

\begin{theorem}
  Let $Max_G=\{x_1,...,x_m\}$ be an essential cyclic set. If $icn(G)=1$ then $rc(\Gamma_G^e)=1$ if only if $m=1$. In particular, if $|InMax_G|=1$ then $rc(\Gamma_G^e)=1$ if only if $G\cong\mathbb{Z}_2$.
\end{theorem}
\begin{proof}
  By \ref{prop1}, \ref{prop23} and \ref{prop27}.
\end{proof}

\begin{theorem}
  Let $Max_G=\{x_1,...,x_m\}$ be an essential cyclic set. If $icn(G)=1$ then $rc(\Gamma_G^e)=2$ if only if $G$ has an awning.
\end{theorem}
\begin{proof}
  By \ref{star}, \ref{cor15} and \ref{prop}.
\end{proof}

\begin{theorem}
  Let $Max_G=\{x_1,...,x_m\}$ be a essential cyclic set with $m\geq 3$. If $icn(G)=1$ then $rc(\Gamma_G^e)=3$ if only if $G$ has not an awning.
\end{theorem}
\begin{proof}
By \ref{prop221}.
\end{proof}
\begin{theorem}
Let $Max_G=\{x_1,...,x_m\}$ be a essential cyclic set with $m\geq 4$. If $icn(G)=2$, then $rc(\Gamma_G^e)=3$.
\end{theorem}
\begin{proof}
 By \ref{prop221} and \ref{prop3}.
\end{proof}

\begin{theorem}
  Let $Max_G=\{x_1,...,x_m\}$ be an essential cyclic set. If $icn(G)\geq 3$, then $rc(\Gamma_G^e)=|InMax_G|$.
\end{theorem}
\begin{proof}
  By \ref{prop221} and \ref{prop27}.
\end{proof}

\begin{theorem} Let $Max_G=\{x_1,...,x_m\}$ be an essential cyclic set with $icn(G)=0$, then
$rc(\Gamma_G^e)=
\begin{cases}
  1, & \mbox{if only if $G$ is a cyclic group }  \\
  2, & \mbox{if only if $G$ has an awning and $G$ is not cyclic.} \\
  3, & \mbox{iff $G$ has not an awning.}.
\end{cases}$
\end{theorem}
\begin{proof}
  \textbf{Case 1} By \ref{prop23}.\\
  \textbf{Case 2} By \ref{prop}.\\
  \textbf{Case 3} By \ref{remawn}, \ref{prop221}.
\end{proof}



\begin{thebibliography}{10}


\bibitem{Aalipour} G. Aalipour, S. Akbari, P. J. Cameron, R. Nikandish and F. Shaveisi. \newblock On the structure of the power graph and
the enhanced power graph of a group. \newblock {\em The Electronic Journal of Combinatorics}, 24(3), \#P3.16, 2017.

\bibitem{Abe} S. Abe and N. Iiyori. \newblock A generalization of prime graphs of finite groups.\newblock {\em Hokkaido Math.
J.}, 29(2):391–-407, 2000.

\bibitem{AndersonLivings} D. F. Anderson, P. S, Livingston. \newblock  The zero-divisor graph of a commutative ring.
\newblock {\em J. Algebra.} 217:434--447, 1999.


\bibitem{Atani} S. E. Atani.\newblock  A ideal based zero divisor graph of a commutative semiring.\newblock {\em Glasnik Matematicki}. 44(64):141--153, 2009.



\bibitem{Bera-Bhuniya} S. Bera, A. K. Bhuniya. \newblock On some properties of enhanced power graph.  \newblock {\em arXiv:1606.03209v1}, 2016.

\bibitem{Bera-Bhuniya2} S. Bera, A. K. Bhuniya. \newblock Normal subgroup based power graph of a finite Group.  \newblock {\em Communications in Algebra}, 45 (8): 3251--3259, 2017.

\bibitem{Chakrabarty} I. Chakrabarty, S. Ghosh, M. K. Sen. \newblock Undirected power graphs of semigroups. \newblock {\em Semigroup Forum}. 78:410--426, 2009.

\bibitem{Chartrand} G. Chartrand, G.L. Johns, K.A. McKeon, P. Zhang. \newblock Rainbow connection in graphs. \newblock Math. Bohem. 133 85-98, 2008.

\bibitem{Diestel} R. Diestel. \newblock Graph theory. \newblock {\em volume 173 of Graduate Texts in Mathematics. Springer-Verlag,
Berlin, third edition}, 2005.






\bibitem{Ma} X. Ma, M. Feng, and  K. Wang. \newblock The Rainbow Connection Number of the Power Graph of a Finite Group. \newblock {\em Graphs and Combinatorics}. 32: 1495, 2016

\bibitem{Redmond} S. P. Redmond.\newblock  An ideal-based zero divisor graph of a commutative ring. \newblock {\em Communication
in algebra}. 31:4425--4443, 2003.

\bibitem{Willians} J. S. Williams.\newblock  Prime graph components of finite groups. \newblock {\em J. Algebra}, 69:487–-513, 1981.

\end{thebibliography}

\end{document}